\documentclass{article}%
\usepackage{amsfonts}
\usepackage{amsmath}
\usepackage{amssymb}
\usepackage{graphicx}%
\providecommand{\U}[1]{\protect\rule{.1in}{.1in}}
\newtheorem{theorem}{Theorem} [section]

\newtheorem{definition}[theorem]{Definition}

\newtheorem{proposition}[theorem]{Proposition}
\newtheorem{remark}[theorem]{Remark}

\newenvironment{proof}[1][Proof]{\noindent\textbf{#1.} }{\ \rule{0.5em}{0.5em}}
\begin{document}

\title{Generalization of the basis theorem for alternating groups}
\author{Robert Shwartz \\ Ariel University \\ ISRAEL \\robertsh@ariel.ac.il
\and Linoy Fanrazi \\ Ariel University \\ ISRAEL \\ linoy.fanrazi@msmail.ariel.ac.il
\and Sawsan Khazkeia \\ Ariel University \\ ISRAEL \\ sawsan.khazkeia@msmail.ariel.ac.il}
\date{}
\maketitle
\begin{abstract}
\noindent There were defined by R Shwartz OGS for non-abelian groups,
as an interesting generalization of the basis of finite abelian groups. The definition of OGS
states that that every element of a group has a unique presentation as a product of
some powers of the OGS, in a specific given order. In case of the symmetric groups $S_{n}$ there is a
paper of R. Shwartz, which demonstrates a strong
connection between the OGS and the standard Coxeter presentation of the symmetric group.
The OGS presentation helps us to find the Coxeter length and the descent set of an arbitrary
element of the symmetric group. Therefore, it motivates us to generalize the OGS for the
alternating subgroup of the symmetric group, which we define in this paper. We generalize also
the exchange laws for the alternating subgroup, and we will show some interesting properties
of it.
\end{abstract}

\section{Introduction}

The fundamental theorem of finitely generated abelian groups states the following:
Let $A$ be a finitely generated abelian group, then there exists generators $a_{1}, a_{2}, \ldots a_{n}$, such that every element $a$ in $A$ has a unique presentation of a form:
$$g=a_{1}^{i_{1}}\cdot a_{2}^{i_{2}}\cdots a_{n}^{i_{n}},$$
where, $i_{1}, i_{2}, \ldots, i_{n}$ are $n$ integers such that for  $1\leq k\leq n$, $0\leq i_{k}<|g_{k}|$, where $a_{k}$ has a finite order of $|a_{k}|$ in $A$, and $i_{k}\in \mathbb{Z}$, where $a_{k}$ has infinite order in $A$.
Where, the meaning of the theorem is that every abelian group $A$ is direct sum of finitely many cyclic subgroup $A_{i}$ (where $1\leq i\leq k$), for some $k\in \mathbb{N}$.

\begin{definition}\label{ogs}
Let $G$ be a non-abelian group. The ordered sequence of $n$ elements $\langle g_{1}, g_{2}, \ldots, g_{n}\rangle$ is called an $Ordered ~~ Generating ~~ System$ of the group $G$ or by shortened notation, $OGS(G)$, if every element $g\in G$ has a unique presentation in the a form
$$g=g_{1}^{i_{1}}\cdot g_{2}^{i_{2}}\cdots g_{n}^{i_{n}},$$
where, $i_{1}, i_{2}, \ldots, i_{n}$ are $n$ integers such that for  $1\leq k\leq n$, $0\leq i_{k}<r_{k}$, where  $r_{k} | |g_{k}|$  in case the order of $g_{k}$ is finite in $G$, or   $i_{k}\in \mathbb{Z}$, in case  $g_{k}$ has infinite order in $G$.
The mentioned canonical form is called $OGS$ canonical form.  For every $q>p$, $1\leq x_{q}<r_{q}$, and $1\leq x_{p}<r_{p}$ the relation
$$g_{q}^{x_{q}}\cdot g_{p}^{x_{p}} = g_{1}^{i_{1}}\cdot g_{2}^{i_{2}}\cdots g_{n}^{i_{n}},$$
is called exchange law.
\end{definition}
In contrast to finitely generated abelian groups, the existence of an $OGS$ is generally not true for every finitely generated non-abelian group. Even in case of two-generated infinite non-abelian groups it is not too hard to find counter examples. For example, the Baumslag-Solitar groups $BS(m,n)$ \cite{BS}, where $m\neq \pm1$ or $n\neq \pm1$, or most of the cases of the one-relator free product of a finite cyclic group generated by $a$, with a finite two-generated group generated by $b, c$ with the relation $a^{2}\cdot b\cdot a\cdot c=1$ \cite{S}, do not have an $OGS$. Even the question of the existence of an $OGS$ for a general finite non-abelian group is still open. Moreover, contrary to the abelian case where the exchange law is just $g_{q}\cdot g_{p}=g_{p}\cdot g_{q}$, in most of the cases of non-abelian groups with the existence of an $OGS$, the exchange laws are very complicated.  Although there are some specific non-abelian groups where the exchange laws are very convenient and have very interesting properties. A very good example of it is the symmetric group $S_{n}$. In 2001, Adin and Roichman \cite{AR} introduced a presentation of an $OGS$ canonical form  for the symmetric group $S_n$, for the hyperoctahedral group $B_n$, and for the wreath product $\mathbb{Z}_{m}\wr S_{n}$. Adin and Roichman proved that for every element of $S_n$ presented in the standard $OGS$ canonical form, the sum of the exponents of the $OGS$ equals the major-index of the permutation. Moreover, by using an $OGS$ canonical form, Adin and Roichman generalized the theorem of MacMahon \cite{Mac} to the $B$-type Coxeter group, and to the wreath product $\mathbb{Z}_{m}\wr S_{n}$. A few years later, that $OGS$ canonical form was generalized for complex reflection groups by Shwartz, Adin and Roichman \cite{SAR}. Recently, Shwartz \cite{Sogs} significantly extended the results of \cite{AR}, \cite{SAR}, where the $OGS$ of $S_{n}$ is strongly connected to the Coxeter length and to the descent set of the elements. Moreover, in \cite{Sogs}, there are described the exchange laws for the $OGS$ canonical forms of the symmetric group $S_n$, which have very interesting and surprising properties.
In this paper we try to generalize the results of \cite{Sogs} to the alternating subgroup $Alt_{n}$ of the symmetric group $S_{n}$. Since $Alt_{n}$ is a simple group, and the only normal subgroup of the symmetric group $S_{n}$ for $n\geq 5$, studying the generalization of the $OGS$ of $S_{n}$ to its subgroup $Alt_{n}$, helps us to understand connections between the property of a group and its subgroup, and helps us to study some unique properties of the family of the simple groups $Alt_{n}$ as well.
Therefore, first we recall the notations of permutations which we use in this paper, the $OGS$ of $S_{n}$ and the corresponding exchange laws, from \cite{Sogs}.

\begin{definition}\label{sn}
Let $S_n$ be the symmetric group on $n$ elements, then :
\begin{itemize}
\item The symmetric group $S_n$ is an $n-1$ generated simply-laced Coxeter group which has the presentation of: $$\langle s_1, s_2, \ldots, s_{n-1} | s_i^{2}=1, ~~ (s_i\cdot s_{i+1})^{3}=1, ~~(s_i\cdot s_j)^2=1 ~~for ~~|i-j|\geq 2\rangle;$$
\item The group $S_n$ can be considered as the permutation group on $n$ elements. A permutation $\pi\in S_n$ is denoted by $[\pi(1);\pi(2);\ldots;\pi(n)]$ (i.e., $\pi=[2;4;1;3]$ is a permutation in $S_{4}$ which satisfies $\pi(1)=2$, $\pi(2)=4$, $\pi(3)=1$, and $\pi(4)=3$);
\item Every permutation $\pi\in S_n$ can be presented in a cyclic notation, as a product of disjoint cycles of the form $(i_1, ~i_2, ~\ldots, ~i_m)$, which means $\pi(i_{k})=i_{k+1}$, for $1\leq k\leq m-1$, and $\pi(i_{m})=i_{1}$
    (i.e., The cyclic notation of $\pi=[3;4;1;5;2]$ in $S_5$, is $(1, ~3)(2, ~4, ~5)$);
\item The Coxeter generator $s_i$ can be considered the permutation which exchanges the element $i$ with the element $i+1$, i.e., the transposition $(i, i+1)$;
\item We consider multiplication of permutations in left to right order; i.e., for every $\pi_1, \pi_2\in S_n$, $\pi_1\cdot \pi_2 (i)=\pi_2(j)$, where, $\pi_1(i)=j$ (in contrary to the notation in \cite{AR} where Adin, Roichman, and other people have considered right to left multiplication of permutations);
\end{itemize}
\end{definition}

\begin{theorem}\label{canonical-sn}
Let $S_n$ be the symmetric group on $n$ elements. For every $2\leq m\leq n$, define $t_{m}$ to be the product $\prod_{j=1}^{m-1}s_{j}$. The element $t_{m}$ is the permutation $[m;1;\ldots;m-1]$, which is the $m$-cycle $(m, ~m-1, ~\ldots, ~1)$ in the cyclic notation of the permutation. Then, the elements $t_{n}, t_{n-1}, \ldots, t_{2}$ generates $S_n$, and every element of $S_n$ has a unique presentation in the following $OGS$ canonical form:

$$t_{2}^{i_{2}}\cdot t_{3}^{i_{3}}\cdots t_{n}^{i_{n}},~~~ where ~~~0\leq i_{k}<k ~~~for ~~~2\leq k\leq n$$
\end{theorem}

\begin{proposition}\label{exchange}
The following holds:
\\

For transforming  the element $t_{q}^{i_{q}}\cdot t_{p}^{i_{p}}$  ($p<q$) onto the $OGS$ canonical form\\ $t_{2}^{i_{2}}\cdot t_{3}^{i_{3}}\cdots t_{n}^{i_{n}}$, i.e., according to the  standard $OGS$, one needs to use the following exchange laws:

 \[ t_{q}^{i_{q}}\cdot t_{p}^{i_{p}}=\begin{cases}
t_{i_{q}+i_{p}}^{i_q}\cdot t_{p+i_{q}}^{i_{p}}\cdot t_{q}^{i_{q}}  & q-i_{q}\geq p \\
\\
t_{i_{q}}^{p+i_{q}-q}\cdot t_{i_{q}+i_{p}}^{q-p}\cdot t_{q}^{i_{q}+i_{p}} & i_{p}\leq q-i_{q}\leq p \\
\\
t_{p+i_{q}-q}^{i_{q}+i_{p}-q}\cdot t_{i_{q}}^{p-i_{p}}\cdot t_{q}^{i_{q}+i_{p}-p}  & q-i_{q}\leq i_{p}.
\end{cases}
\]
\end{proposition}

\begin{definition}\label{alt-n}
Let $Alt_{n}$ be the subgroup of $S_{n}$ which contains all the elements which can be written as a product of even number of transpositions $s_i$.
\end{definition}

Obviously, $Alt_{n}$ is a subgroup of index $2$ of the symmetric group $S_{n}$, with $|Alt_{n}|=\frac{n!}{2}$.

\section{The $OGS$ of $Alt_{n}$}

In this section we find an $OGS$ for $Alt_{n}$,  which can be considered as a generalization of the $OGS$ of $S_{n}$ as defined in Theorem \ref{canonical-sn} in \cite{Sogs}.

\begin{proposition}\label{t-a-n}
Consider $t_{m}$ as defined in Theorem \ref{canonical-sn}, for $2\leq m\leq n$. Then, $t_{m}\in Alt_{n}$ if and only if $m$ is odd.
\end{proposition}

\begin{proof}
Since $t_{m}=\prod_{j=1}^{m-1}s_{j}$, we have $t_{m}$ is a product of $m-1$ transpositions. Hence, by Definition \ref{alt-n}, $t_{m}\in Alt_{n}$ if and only if $m$ is odd.
\end{proof}

\begin{definition}\label{uv}
For every integer $r$ such that $2\leq r\leq \frac{n}{2}$, let define $u_{2r}$ and $v_{2r}$ as follows:
\begin{itemize}
\item $u_{2r}=t_{2r-2}\cdot s_{2r-1}=(2r-2, 2r-1, \ldots, 1)(2r-1, 2r)$;
\item $v_{2r}=t_{2r}^{2}=(2r, 2r-2, \ldots, 2)(2r-1, 2r-3, \ldots, 1)$.
\end{itemize}
\end{definition}

\begin{remark}\label{order-uv}
By Definition \ref{uv}, $u_{2r}$ and $v_{2r}$ in $Alt_{n}$ satisfy the following properties:
$$|u_{2r}|=2r-2, ~~~~~~~ |v_{2r}|=r, ~~~~~~~v_{2r-2}=u_{2r}^{2}.$$
\end{remark}

\begin{definition}\label{n-fixed}
Let $Alt_{n-1}^{\circ}$ be the subgroup of $Alt_{n}$ which contains all the even permutations such that $n$ is a fixed point.
\end{definition}

\begin{remark}\label{index-n-fixed}
Obviously, $Alt_{n-1}^{\circ}$ is isomorphic to the alternating group $Alt_{n-1}$, and the index of $Alt_{n-1}^{\circ}$ is $n$ in $Alt_{n}$.
\end{remark}

\begin{theorem}\label{main}
Let $Alt_{n}$ be the alternating group on $n$ elements for $n\geq 3$.
\begin{enumerate}
\item Let $\ell$ be an integer greater or equal to $1$ and let $n=2\ell+1$, then the elements:
$$t_{3}, u_{4}, v_{4}, \ldots, t_{2\ell-1}, u_{2\ell}, v_{2\ell}, t_{2\ell+1}$$
form an $OGS$ for $Alt_{n}$,
such that every element $g\in Alt_{n}$ has a unique presentation of the form:
$$g=t_{3}^{i_{3}}\cdot \prod_{r=2}^{\ell}u_{2r}^{j_{2r}}\cdot v_{2r}^{k_{2r}}\cdot t_{2r+1}^{i_{2r+1}},$$
where

$$0\leq i_{3}<3, ~~~~~~~ 0\leq i_{2r+1}<2r+1, ~~~~~~~~ 0\leq j_{2r}<2, ~~~~~~~~ 0\leq k_{2r}<r,$$
for $2\leq r\leq \ell$;

\item Let $\ell$ be an integer greater or equal to $2$ and let $n=2\ell$. Then the elements:
$$t_{3}, u_{4}, v_{4}, \ldots, t_{2\ell-1}, u_{2\ell}, v_{2\ell}$$
form an $OGS$ for $Alt_{n}$,
such that every element $g\in Alt_{n}$ has a unique presentation of the form:
$$g=\prod_{r=2}^{\ell}t_{2r-1}^{i_{2r-1}}\cdot u_{2r}^{j_{2r}}\cdot v_{2r}^{k_{2r}},$$
where
$$0\leq i_{2r-1}<2r-1, ~~~~~~~~ 0\leq j_{2r}<2, ~~~~~~~~ 0\leq k_{2r}<r,$$
for $2\leq r\leq \ell$.
\end{enumerate}
\end{theorem}

\begin{proof}
The proof is by induction on $n$. Since $Alt_{3}$ is cyclic of order $3$, the theorem holds obviously. Now, consider $n=4$. The group $Alt_{4}$ can be considered as the even permutation on the set $\{1, 2, 3, 4\}$. Since $|Alt_{4}|=12$, and its subgroup $Alt_{3}^{\circ}$ is cyclic of order $3$, which is generated by $t_{3}$, the index of $Alt_{3}^{\circ}$ in $Alt_{4}$ is $4$. Moreover, since $Alt_{3}^{\circ}$ is the subgroup of $Alt_{4}$, which contains all the even permutations on $\{1, 2, 3, 4\}$ such that $4$ is a fixed point, a right coset of it in $Alt_{4}$ is determined by the image of $4$ as a permutation.
Therefore, since $u_{4}=(1, 2)(3, 4)$ and $v_{4}=(1, 3)(2, 4)$, we conclude the following:
\begin{itemize}
\item The subgroup $Alt_{3}^{\circ}$ are all the even permutation $\pi$ on $\{1, 2, 3, 4\}$ such that $\pi(4)=4$;
\item The coset $Alt_{3}^{\circ}\cdot u_{4}$ are all the even permutation $\pi$ on $\{1, 2, 3, 4\}$ such that $\pi(4)=3$;
\item The coset $Alt_{3}^{\circ}\cdot v_{4}$ are all the even permutation $\pi$ on $\{1, 2, 3, 4\}$ such that $\pi(4)=2$;
\item The coset $Alt_{3}^{\circ}\cdot u_{4}\cdot v_{4}$ are all the even permutation $\pi$ on $\{1, 2, 3, 4\}$ such that $\pi(4)=1$.
\end{itemize}
Hence, every element of $Alt_{4}$ has a unique presentation of a form:
$$t_{3}^{i_{3}}\cdot u_{4}^{j_{4}}\cdot v_{4}^{k_{4}},$$
where
$$0\leq i_{3}<3 ~~~~~~~ 0\leq j_{4}<2, ~~~~~~~ 0\leq k_{4}<2.$$

Now, assume by induction the theorem holds for $Alt_{m}$ for every $m\leq 2\ell$ for some $\ell\geq 2$, and we prove it for $Alt_{2\ell+1}$. By Definition \ref{n-fixed}, the index of the subgroup $Alt_{2\ell}^{\circ}$ is  $2\ell+1$ in $Alt_{2\ell+1}$ which contains all the even permutation $\pi$ on the set $\{1, 2, \ldots, 2\ell+1\}$ such that $\pi(2\ell+1)=2\ell+1$. Therefore, the $2\ell+1$ different right cosets of $Alt_{2\ell}^{\circ}$ on $Alt_{2\ell+1}$ are determined by the image of $2\ell+1$ as a permutation on the set $\{1, 2, \ldots, 2\ell+1\}$. Notice, $t_{2\ell+1}=(2\ell+1, 2\ell, \ldots, 1)$. Therefore, the following holds:
\begin{itemize}
\item The coset $Alt_{2\ell}^{\circ}\cdot t_{2\ell+1}^{i_{2\ell+1}}$ contains all the even permutations $\pi$, such that $$\pi(2\ell+1)=2\ell+1-i_{2\ell+1}.$$
\end{itemize}
Hence, every element of $Alt_{2\ell+1}$ has a unique presentation of the form $$g_{2\ell}\cdot t_{2\ell+1}^{i_{2\ell+1}},$$ such that $g_{2\ell}\in Alt_{2\ell}^{\circ}$ and $0\leq i_{2\ell+1}<2\ell+1$.

Now, we prove the theorem for $Alt_{2\ell+2}$. By Definition \ref{n-fixed}, the index of the subgroup $Alt_{2\ell+1}^{\circ}$ is  $2\ell+2$ in $Alt_{2\ell
+2}$ which contains all the even permutation $\pi$ on the set $\{1, 2, \ldots, 2\ell+2\}$ such that $\pi(2\ell+2)=2\ell+2$. Therefore, the $2\ell+2$ different right cosets of $Alt_{2\ell+1}^{\circ}$ on $Alt_{2\ell+2}$ are determined by the image of $2\ell+2$ as a permutation on the set $\{1, 2, \ldots, 2\ell+2\}$.
Since $u_{2\ell+2}=(2\ell, 2\ell-1, \ldots, 1)(2\ell+1, 2\ell+2)$ and $v_{2\ell+2}=(2\ell+2, 2\ell, \ldots, 2)(2\ell+1, 2\ell-1, \ldots, 1)$, the following holds:
\begin{itemize}
\item The coset $Alt_{2\ell+1}^{\circ}\cdot v_{2\ell+2}^{k_{2\ell+2}}$ contains all the even permutations $\pi$, such that $\pi(2\ell+2)$ is even between $0$ and $2\ell$, namely $$\pi(2\ell+2)=2\ell+2-2k_{2\ell+2};$$
\item The coset $Alt_{2\ell+1}^{\circ}\cdot u_{2\ell+2}\cdot v_{2\ell+2}^{k_{2\ell+2}}$ contains all the even permutations $\pi$, such that $\pi(2\ell+2)$ is odd between $1$ and $2\ell+1$, namely $$\pi(2\ell+2)=2\ell+1-2k_{2\ell+2}.$$
\end{itemize}
Hence, every element of $Alt_{2\ell+2}$ has a unique presentation of the form $$g_{2\ell+1}\cdot u_{2\ell+2}^{j_{2\ell+2}}\cdot v_{2\ell+2}^{k_{2\ell+2}},$$ such that $g_{2\ell+1}\in Alt_{2\ell+1}^{\circ}$, ~$0\leq j_{2\ell+1}<2$, ~$0\leq k_{2\ell+1}<\ell$.

Thus the theorem holds for every $n\geq 3$.
\end{proof}

\section{exchange laws for the $OGS$ of $Alt_{n}$}

In \cite{Sogs}, there are described very interesting exchange laws for the $OGS$ of the symmetric group $S_{n}$, which we mention at the first section in Proposition \ref{exchange}. Thus, it is interesting to see which exchange laws we get by the $OGS$ of $Alt_{n}$ which we have described in Theorem \ref{main}.
Since for $n=3$, the group $Alt_{3}$ is cyclic, the exchange laws in that case are trivial. Therefore, we start with $n=4$. By Theorem \ref{main}, every element in $Alt_{4}$ has a unique presentation of a form 
$$t_{3}^{i_{3}}\cdot u_{4}^{j_{4}}\cdot v_{4}^{k_{4}},$$
where 
$$t_{3}=(1, 3, 2), ~~~~~~~ u_{4}=(1, 2)(3, 4), ~~~~~~~ v_{4}=(1, 3)(2, 4),$$
and
$$i_{3}\in \{0, 1, 2\}, ~~~~~~~ j_{4}\in \{0, 1\}, ~~~~~~~ k_{4}\in \{0, 1\}.$$
Hence, by multiplication of permutations, the following exchange laws holds in $Alt_{4}$:
\begin{itemize}
\item $u_{4}\cdot t_{3}=t_{3}\cdot v_{4}$;
\item $u_{4}\cdot t_{3}^{2}=t_{3}^{2}\cdot u_{4}\cdot v_{4}$;
\item $v_{4}\cdot t_{3}=t_{3}\cdot u_{4}\cdot v_{4}$;
\item $v_{4}\cdot t_{3}^{2}=t_{3}^{2}\cdot u_{4}$;
\item $v_{4}\cdot u_{4}=u_{4}\cdot v_{4}$.
\end{itemize}

In the case of $n>4$, the exchange laws for the $OGS$ of $Alt_{n}$ which described in Theorem \ref{main} are generally much more complicated, but there are some interesting and surprising properties for some special cases of exchange laws, which we describe now.

\begin{proposition}\label{v-exchange}
Consider $Alt_{n}$ for $n\geq 4$. Let $r$ be an integer, such that $2\leq r\leq \frac{n}{2}$, and let $v_{2r}$ be the element of $Alt_{n}$ as defined in Definition \ref{uv}. Then the following exchange laws holds for all integers $p$ and $q$ such that $2\leq p<q\leq \frac{n}{2}$
\[ v_{2q}^{k_{2q}}\cdot v_{2p}^{k_{2p}}=\begin{cases}
v_{2k_{2q}+2k_{2p}}^{k_{2q}}\cdot v_{2p+2k_{2q}}^{k_{2p}}\cdot v_{2q}^{k_{2q}}  & q-k_{2q}\geq p \\
\\
v_{2k_{2q}}^{p+k_{2q}-q}\cdot v_{2k_{2q}+2k_{2p}}^{q-p}\cdot v_{2q}^{k_{2q}+k_{2p}} & k_{2p}\leq q-k_{2q}\leq p \\
\\
v_{2p+2k_{q}-2q}^{k_{2q}+k_{2p}-q}\cdot v_{2k_{2q}}^{p-k_{2p}}\cdot v_{2q}^{k_{2q}+k_{2p}-p}  & q-k_{2q}\leq k_{2p}.
\end{cases}
\]
\end{proposition} 

\begin{proof}
Since by Definition \ref{uv}, $v_{2}=t_{2r}^{2}$, we conclude the proof in the same way as the proof of Proposition \ref{exchange} in \cite{Sogs}.
\end{proof}

\begin{proposition}\label{rel-exchange-alt}
Consider the symmetric group $S_{n}$ and its alternating subgroup $Alt_{n}$ for $n\geq 4$. For every integer $q$ such that $2\leq q\leq n$, let $t_{q}=(q, q-1, \ldots, 1)=\prod_{i=1}^{q-1}s_{i}$, as defined in \cite{Sogs}. For every integer $r$ such that $2\leq r\leq \frac{n}{2}$ let $u_{2r}$ and $v_{2r}$ be the elements of $Alt_{n}$ as defined in Definition \ref{uv}, then the following relations holds in $S_{n}$
\begin{itemize}
\item $t_{2r}^{-1}\cdot t_{2r+2}=t_{2r+1}^{-1}\cdot u_{2r+2}$;
\item $t_{2r}\cdot t_{2r+2}=v_{2r}\cdot t_{2r+1}^{-1}\cdot u_{2r+2}$.
\end{itemize}
More generally, for every $r'>r$, the following holds
\begin{itemize}
\item $t_{2r}^{-1}\cdot t_{2r'}=\prod_{i=2r+1}^{2r'-1}t_{i}\cdot u_{i+1}$;
\item $t_{2r}\cdot t_{2r'}=v_{2r}\cdot \prod_{i=2r+1}^{2r'-1}t_{i}\cdot u_{i+1}$.
\end{itemize}
\end{proposition}

\begin{proof}
By using the definition of $t_{q}$ from \cite{Sogs}, and the definitions of $u_{2r}$ and $v_{2r}$ from Definition \ref{uv}, the following holds:
\begin{itemize}
\item $t_{2r}^{-1}\cdot t_{2r+2}=s_{2r}\cdot s_{2r+1}=\left(s_{2r}\cdot s_{2r-1}\cdots s_{1}\right)\cdot \left(s_{1}\cdot s_{2}\cdot s_{2r-1}\right)\cdot s_{2r+1}=t_{2r+1}^{-1}\cdot u_{2r+2}$;
\item $t_{2r}\cdot t_{2r+2}=t_{2r}^{2}\cdot t_{2r}^{-1}\cdot t_{2r+2}=v_{2r}\cdot t_{2r+1}^{-1}\cdot u_{2r+2}$;    
\end{itemize}
Then by using
\begin{itemize}
\item $t_{2r}^{-1}\cdot t_{2r'}=\prod_{i=r}^{r'-1}t_{2i}^{-1}\cdot t_{2i+2}$;
\item $t_{2r}\cdot t_{2r'}=v_{2r}\cdot \prod_{i=r}^{r'-1}t_{2i}^{-1}\cdot t_{2i+2}$,
\end{itemize}
we conclude the rest of the results the proposition.
\end{proof}

\begin{proposition}
The following exchange laws are satisfied in the alternating group $Alt_{n}$.
\begin{itemize}
\item $v_{2r}\cdot u_{2r}=\prod_{i=2}^{r-1}\left(u_{2i}\cdot t_{2i+1}^{-1}\right)\cdot u_{2r}\cdot v_{2r}$, for every $r\geq 3$;
\item $t_{2r'-1}\cdot t_{2r-1}=\prod_{i=2}^{r}\left(t_{2i-1}^{-1}\cdot u_{2i}\right)\cdot t_{2r'-1}$, for every $r'>r\geq 2$.
\end{itemize}
\end{proposition}

\begin{proof}
First, notice the following observations. 
By Definition \ref{uv}, $u_{2r}=t_{2r-2}\cdot s_{2r-1}$. Hence, by \cite{Sogs}, $u_{2r}=\prod_{i=1}^{2r-3}s_i\cdot s_{2r-1}$. By Definition \ref{uv}, \\ $v_{2r}=(2r, 2r-2, \ldots, 2)(2r-1, 2r-3, \ldots, 1)$. Therefore
$$v_{2r}\cdot u_{2r}\cdot v_{2r}^{-1}=s_{1}\cdot \prod_{i=3}^{2r-1}s_{i}=\left(s_{1}\cdot s_{3}\right)\cdot \prod_{i=4}^{2r-1}s_{i}.$$
Now, notice the following observations:
\begin{itemize}
\item By Definition \ref{uv}, $s_{1}\cdot s_{3}=u_{4}$;
\item since $\prod_{i=4}^{2r-1}s_{i}=\left(s_{3}\cdot s_{2}\cdot s_{1}\right)\cdot \prod_{i=1}^{2r-1}s_{i}$, we conclude by \cite{Sogs}, $\prod_{i=4}^{2r-1}s_{i}=t_{4}^{-1}\cdot t_{2r}$. Then by using Proposition \ref{rel-exchange-alt}, $t_{4}^{-1}\cdot t_{2r}=\prod_{i=3}^{r}t_{2i-1}^{-1}\cdot u_{2i}$.
\end{itemize}
Hence, 
$$v_{2r}\cdot u_{2r}=v_{2r}\cdot u_{2r}\cdot v_{2r}^{-1}\cdot v_{2r}=u_{4}\cdot \prod_{i=3}^{r}\left(t_{2i-1}^{-1}\cdot u_{2i}\right)\cdot v_{2r}=\prod_{i=2}^{r-1}\left(u_{2i}\cdot t_{2i+1}^{-1}\right)\cdot u_{2r}\cdot v_{2r},$$ for every $r\geq 3$

Now, we turn to the second part of the proposition. Consider the expression \\ $t_{2r'-1}\cdot t_{2r-1}$, where $r'>r\geq 2$. By Proposition \ref{exchange}, $t_{2r'-1}\cdot t_{2r-1}=t_{2}\cdot t_{2r}\cdot t_{2r'-1}$. Now, by using Proposition \ref{rel-exchange-alt}, ~$t_{2}\cdot t_{2r}=\prod_{i=2}^{r}t_{2i-1}\cdot u_{2i}$. Hence, 
$$t_{2r'-1}\cdot t_{2r-1}=\prod_{i=2}^{r}\left(t_{2i-1}^{-1}\cdot u_{2i}\right)\cdot t_{2r'-1},$$ for every $r'>r\geq 2$.
\end{proof}

\section{Conclusion and future plans}
The paper generalize the standard $OGS$ of the symmetric group $S_{n}$ as defined in \cite{Sogs} to its alternating subgroup $Alt_{n}$. In the paper we found some interesting properties of the exchange laws as well. Therefore, the results of the paper motivate us to find more connections between the $OGS$ and the presentation of the elements of the alternating group in a form of permutation, similarly to the results in the case of the symmetric groups in \cite{Sogs}. Other direction of further research is studying particular subgroups of the alternating group $Alt_{n}$ by using the properties of the $OGS$ which we described in this paper.

\end{document}